\documentclass[a4paper, 12pt]{article}
\usepackage{amssymb,amsbsy,float}
\usepackage{graphicx}
\usepackage{amsfonts}
\usepackage{latexsym}
\usepackage[centertags]{amsmath}
\usepackage{amsthm}
\usepackage{newlfont}
\usepackage[text={15cm, 22cm}]{geometry}
\newtheorem{theorem}{Theorem}

\newtheorem*{claim*}{Claim}

\newtheorem{corollary}[theorem]{Corollary}

\newtheorem{definition}[theorem]{Definition}
\newtheorem{example}[theorem]{Example}

\newtheorem{lemma}[theorem]{Lemma}

\newtheorem{proposition}[theorem]{Proposition}

\newcommand{\R}{\mbox{$\mathbb{R}$}}%
\newcommand{\N}{\mbox{$\mathbb{N}$}}%
\begin{document}
\title{Properties and applications of dual reduction}
\author{Yannick Viossat
\thanks{This is a revised version of the cahier du laboratoire d'\' econom\' etrie de l'Ecole polytechnique 2003-31. The latter was written during my PhD, under the supervision of Sylvain Sorin. I am grateful to him, Bernhard von Stengel, Fran{\c c}oise Forges, Ehud Lehrer and seminar audiences.}}
\date{CEREMADE, Universit\' e Paris-Dauphine,\vspace{0.1cm}\\
 Place du mar\' echal de Lattre de Tassigny, 75016 Paris, France.\vspace{0.3cm}\\
 E-mail: viossat@ceremade.dauphine.fr.}
\maketitle
\abstract{The dual reduction process, introduced by Myerson, allows to reduce a finite game into a smaller dimensional game such that any equilibrium of the reduced game is an equilibrium of the original game. This holds both for Nash equilibrium and correlated equilibrium. We present examples of applications of dual reduction and argue that this is a useful tool to study Nash equilibria and correlated equilibria. We then investigate its properties.\\


}
%
\newpage
\section{Introduction}
Dual reduction is a reduction process for finite games defined by
Roger Myerson (1997), which generalizes elimination of dominated
strategies. It takes its roots in the proofs of existence of
correlated equilibria of Hart and Schmeidler (1989) and Nau and
McCardle (1990).  Dual reduction's main property is that any Nash or
correlated equilibrium of the reduced game is an equilibrium of the
original game. Moreover, by iterative dual reduction, any finite
game may be reduced to an elementary game; that is, to a game in
which all incentive constraints defining correlated equilibria can
be satisfied as strict inequalities in a correlated equilibrium.
Myerson (1997) also showed that, while some games may be reduced in
several ways, this ambiguity is alleviated if we focus on a
specific class of dual reductions, called full dual reductions.

The aim of this article is to show that dual reduction is a powerful
tool to study Nash equilibria and correlated equilibria, and to
study its properties. It is organized as follows: in the next
section, we recall the basics of dual reduction.
Section \ref{sec:app} gives examples of applications of dual
reduction:  we first give an elementary proof of the fact that a
unique correlated  equilibrium is a Nash equilibrium. The proof is
elementary in that it relies entirely on linear programming. We then
show that dual reduction allows to prove the existence of Nash
equilibria with special properties and to study the geometry of the
set of correlated equilibria. In section \ref{sec:prop}, the
properties of dual reduction are investigated. For instance, we show
that rescaling the payoffs does not change the ways in which a game
may be reduced; that in all full dual reductions, all strategies
that have probability zero in all correlated equilibria are
eliminated ; and that any zero-sum game and any game with a unique
correlated equilibrium is reduced by all full dual reductions into a
game with a unique pure strategy profile. We also show that, in
almost all two-player games, the iterative reduction
process is uniquely defined, as long as we focus on full dual
reductions. Finally, a proof is given in the appendix. Unless stated
otherwise, all references to Myerson are to Myerson (1997).\\

\textbf{Notations.} As Myerson, we denote a finite game in strategic form by
$$\Gamma=(N, (C_i)_{i \in N}, (U_i)_{i \in N})$$
where $N$ is the set of players, $C_i$ the set of pure strategies of
player $i$, and $U_i : \times_{j \in N} C_j \to \R$ the utility
function of player $i$. The set of pure strategy profiles is denoted
by $C=\times_{j \in N} C_j$. For each player $i$, we let
$C_{-i}=\times_{j \in N\backslash\{i\}} C_j$. If $c=(c_j)_{j \in N}$
is a pure strategy profile and $d_i$ a pure strategy of player $i$,
we let $(c_{-i},d_i)$ denote the pure strategy profile that differs
from $c$ only in that the $i$-component is $d_i$.

For any finite set $S$, we let $|S|$ denote its cardinal and
$\Delta(S)$ the set of probability distributions over $S$. A
correlated strategy of the players in $N$ is an element of
$\Delta(C)$. Thus, $\mu=(\mu(c))_{c \in C}$ is a correlated strategy
if and only if $\mu(c )\geq 0$ for any $c$ in $C$ and $\sum_{c \in
C} \mu(c)=1$.

Let $\mu$ be a correlated strategy. Assume that before play, a
mediator draws a pure strategy profile $c$ with probability $\mu(c)$
and then privately recommends $c_i$ to player $i$, for every $i$ in
$N$. The correlated strategy $\mu$ is a \emph{correlated equilibrium}
(Aumann, 1974, 1987) if no player has an incentive to deviate
unilaterally from these recommendations. That is, $\mu$ is a
correlated equilibrium if and only if it satisfies the following
incentive constraints:
\begin{equation}
\label{eq:defcor} \sum_{c_{-i} \in C_{-i}} \mu(c)\left[U_i(c_{-i},
d_i)- U_i(c)\right]\leq 0 \quad \quad \forall i \in N, \forall c_i
\in C_i, \forall d_i \in C_i
\end{equation}
For every $i$ in $N$, we let $C_i^c \subset C_i$ denote the set of
pure strategies of player $i$ which have positive marginal
probability in at least one correlated equilibrium.

For any mapping $\alpha_i : C_i \to \Delta(C_i)$ and any $c_i, d_i$
in $C_i$, we may write $\alpha_i \ast c_i$ instead of
$\alpha_i(c_i)$ to denote the image of $c_i$ by this mapping, and
$\alpha_i(d_i|c_i)$ instead of $(\alpha_i \ast c_i) (d_i)$ to denote
the probability of $d_i$ in the mixed strategy $\alpha_i \ast c_i$.
This is for consistency with Myerson.

\section{Basics of dual reduction}
\label{sec:basics}
Unless stated otherwise, all results of this section are due to Myerson.\\

{\bf Dual vectors. } Assume that before play a mediator privately
recommends a pure strategy to each player, who can either obey or
deviate from this recommendation. The behavior of player $i$ can
then be described by a mapping $\alpha_i: C_i \to \Delta(C_i)$,
which associates to every pure strategy $c_i$  the randomized
strategy $\alpha_i \ast c_i$ that she will play if recommended $c_i$.
The mapping $\alpha_i$ may be called a \emph{deviation plan}
for player $i$. If a mediator tries to implement a pure strategy
profile $c$, then player $i$'s gain from deviating unilaterally
according to $\alpha_i$ instead of following the mediator's
recommendation is
\[ D_i(c, \alpha_i):=U_i(c_{-i}, \alpha_i \ast c_i) - U_i(c)\]
Let $\alpha=(\alpha_i)_{i \in N}$ be a profile of deviation plans and
let $D(c, \alpha)$ denote the sum of the above gains over the set of players:
\[D(c, \alpha):=\sum_{i \in N} D_i(c, \alpha_i)= \sum_{i \in N} \left[U_i(c_{-i}, \alpha_i \ast c_i) - U_i(c)\right]
\]

\noindent \emph{Definition. } A \emph{dual vector} is a profile of deviation
plans $\alpha=(\alpha_i)_{i\in N}$ such that $D(c, \alpha) \geq 0$ for every pure strategy profile $c$ in $C$.\\

Dual vectors arise from the linear programming proofs of existence
of correlated equilibria (Hart and Schmeidler, 1989; Nau and
McCardle, 1990). Any game has at least one dual vector. Indeed,
letting $\alpha_i \ast c_i =c_i$ for all $i$ in $N$ and all $c_i$ in
$C_i$ defines a dual vector. We call it the
\emph{trivial dual vector}.\footnote{To see how dual vectors arise from the proofs of
existence of correlated equilibria, consider the auxiliary zero-sum
game in which the maximizer chooses a pure strategy profile $c$ of
$\Gamma$, the minimizer a profile of deviation plans
$\alpha=(\alpha_i)_{i\in N}$, and the payoff is $-D(c,\alpha)$. In
this game (a variant of Hart and Schmeidler's (1989) auxiliary
game), a randomized strategy of the maximizer is a correlated
strategy of $\Gamma$, and it guarantees a payoff of $0$ if and only
if it is a correlated equilibrium. Thus, to prove existence of
correlated equilibria, it suffices to show that the value of the
auxiliary game is at least $0$. In fact, it is exactly $0$ because
the minimizer can guarantee $0$ by choosing the trivial deviation
vector. It follows that dual vectors may be defined as the optimal strategies of
 the minimizer in this auxiliary game. Similarly, Myerson introduces dual
 vectors as the optimal solutions of the dual of a linear program
 whose optimal solutions are the correlated equilibria.}\\

{\bf How to reduce a game using a dual vector. } Let $\alpha$ be a dual vector. The mapping $\alpha_i : C_i \to \Delta(C_i)$  induces  a  Markov chain on $C_i$.
This Markov chain partitions $C_i$ into a set of transient states and disjoint minimal absorbing sets.\footnote{A nonempty subset $B_i$ of $C_i$ is a \emph{minimal absorbing set} if : (i)  for every $c_i$ in $B_i$, $\alpha_i \ast c_i$ has support in $B_i$ and (ii) it contains no nonempty proper subset satisfying (i).} %
Let us say that a mixed strategy $\sigma_i$ of player $i$ is \emph{$\alpha_i$-stationary} if $\alpha_i \ast \sigma_i = \sigma_i$, where
\[\alpha_i \ast \sigma_i= \sum_{c_i \in C_i}  \sigma_i(c_i) (\alpha_i \ast c_i)\]
From the basic theory of Markov chains, it follows that for every minimal absorbing set $B_i \subset C_i$, there is a unique $\alpha_i$-stationary mixed stategy with support in $B_i$. Let $C_i/\alpha_i$ denote the set of such $\alpha_i$-stationary strategies with support in a minimal absorbing set. The $\alpha$-reduced game, denoted by $\Gamma/\alpha$, is the game with the same set of players and the same utility functions as in $\Gamma$, but in which, for every $i$ in $N$, the pure strategy set of player $i$ is $C_i/\alpha_i$:
 \[\Gamma/\alpha=\left\{N, (C_i/\alpha_i)_{i \in N}, (U_i)_{i \in N} \right\}\]
The reduction thus operates by eliminating some strategies and grouping other strategies together.

More precisely, if $c_i \in C_i$ is a transient state for the Markov
chain induced by $\alpha_i$, then $c_i$ is \emph{eliminated}. That
is, $\sigma_i(c_i)=0$ for all $\sigma_i$ in $C_i/\alpha_i$. If $c_i$
is $\alpha_i$-stationary (i.e. $\alpha_i \ast c_i =c_i$), then $c_i$
is \emph{kept} as a pure strategy: $c_i \in C_i/\alpha_i$. Finally,
if $c_i$ is neither transient nor stationary, then $c_i$ is
\emph{grouped} with the other strategies of its minimal absorbing
set, so that: $\exists \sigma_i \in C_i/\alpha_i, \sigma_i \neq c_i
\mbox{ and } \sigma_i(c_i)>0$.\\


\noindent \emph{Definition. } A \emph{dual reduction} of $\Gamma$ is any
$\alpha$-reduced game $\Gamma/\alpha$ where $\alpha$ is a dual
vector. An \emph{iterative dual reduction} of $\Gamma$ is any game
$\Gamma/\alpha_1/\alpha_2/\hdots$ where each $\alpha^k$ is a dual
vector for $\Gamma/\alpha^1/\alpha^2/\hdots/\alpha^{k-1}$. (The expression ``dual reduction" may refer either to a reduced game or to the reduction technique.)\\

Let $C/\alpha=\times_{i \in N} C_i/\alpha_i$ denote the set of pure
strategy profiles of the reduced game $\Gamma/\alpha$. So the set of
correlated strategies of the reduced game is $\Delta(C/\alpha)$.
Any correlated strategy $\mu$ in $C/\alpha$ can be mapped
back to a $\Gamma$-equivalent correlated strategy $\bar{\mu}$ in the
natural way:
\[\bar{\mu}(c)=\sum_{\sigma \in C/\alpha} \mu(\sigma) \sigma(c) \quad \forall c \in C\]
(the mapping $\mu \mapsto \bar{\mu}$ may be shown to be injective). Myerson's main result is that if $\mu$ is a correlated equilibrium of a dual reduction of $\Gamma$, then $\bar{\mu}$ is a correlated equilibrium of $\Gamma$. The same result holds for Nash equilibrium.

The main step of the proof is to show that, if $\sigma$ is a mixed strategy profile such that,
for every $j\neq i$, $\sigma_j$ is $\alpha_j$-stationary,
then $U_i(\sigma) \leq U_i(\sigma_{-i}, \alpha_i \ast \sigma_i)$.
The derivation of this fact will be used in sections \ref{sec:app} and \ref{sec:prop} and so we recall it:  for every $j$ in $N$,
\[\sum_{c \in C}\sigma(c)D_j(c,\alpha_j)=U_j(\sigma_{-j}, \alpha_j \ast \sigma_j) - U_j(\sigma)\]
Indeed, both sides of the equation represent the expected gain for
player $j$ from deviating according to $\alpha_j$, when all other
players obey a mediator who is trying to implement $\sigma$. Summing
over the set of players $N$,  we get
\begin{equation}
\label{eq:rev-variant-lemma1} \sum_{c \in
C}\sigma(c)D(c,\alpha)=\sum_{j\in N} \left[U_j(\sigma_{-j}, \alpha_j
\ast \sigma_j) - U_j(\sigma)\right]
\end{equation}
If $\alpha_j \ast \sigma_j =\sigma_j$ for all $j \neq i$, then (\ref{eq:rev-variant-lemma1}) boils down to %
\begin{equation}
\label{eq2:rev-variant-lemma1} \sum_{c \in
C}\sigma(c)D(c,\alpha)=U_i(\sigma_{-i}, \alpha_i \ast \sigma_i) -
U_i(\sigma)
\end{equation}
Since $\alpha$ is a dual vector, $D(c,\alpha)\geq 0$ for all $c$,
hence the left hand side of (\ref{eq2:rev-variant-lemma1}) is nonnegative. Thus, (\ref{eq2:rev-variant-lemma1}) $\Rightarrow U_i(\sigma) \leq U_i(\sigma_{-i}, \alpha_i \ast \sigma_i)$.\\


{\bf Jeopardization and reduction to elementary games}
A game is \emph{elementary} (Myerson) if it has a correlated equilibrium
$\mu$ which satisfies all incentive constraints with strict
inequality:
\begin{equation}
\label{eq:elem-incentives} \sum_{c \in C} \mu(c)\left[U_i(c_{-i}, d_i) - U_i(c)\right] >0 \quad \forall i \in N, \forall c_i \neq d_i
\end{equation}
If $\mu$ satisfies (\ref{eq:elem-incentives}), then for any
correlated strategy $\mu'$ with full support and for any $\epsilon$
small enough, $(1-\epsilon)\mu + \epsilon \mu'$ is a strict
correlated equilibrium with full support; conversely, any strict
correlated equilibrium with full support satisfies
(\ref{eq:elem-incentives}). So, elementary games can also be defined
as those games having a strict correlated equilibrium with full
support.\footnote{A correlated equilibrium $\mu$ has \emph{full support} if
$\mu(c)>0$ for all $c$ in $C$. It is \emph{strict} if the incentive constraint (\ref{eq:defcor}) is
satisfied with strict inequality, for any $i$ in $N$, any $c_i$ with
positive marginal probability in $\mu$, and any $d_i \neq c_i$.}

Myerson shows that a game may be strictly reduced if and only if it is 
not elementary; this implies that, though many games are not
elementary (e.g., Matching Pennies), any game may be reduced to an
elementary game by iterative dual reduction. The proof is based on
the concept of jeopardization and proposition \ref{prop:M-comp-dv} below: a strong complementary slackness property which we will use extensively.\\

\noindent \emph{Definition.} \quad Let $c_i$ and $d_i$ be pure strategies of player $i$. The strategy $d_i$ \emph{jeopardizes} $c_i$ if,
in every correlated equilibrium $\mu$, the incentive constraint stipulating that player $i$ should not gain by deviating from $c_i$ to $d_i$ is tight. That is,
\[\sum_{c_{-i} \in C_{-i}} \mu(c)\left[U_i(c_{-i}, d_i) - U_i(c)\right] =0\]
%
\begin{proposition}
\label{prop:M-comp-dv}
There exists a dual vector $\alpha$ such that $\alpha_i(d_i|c_i)>0$ if and only if $d_i$ jeopardizes $c_i$.
\end{proposition}

{\bf Full dual reduction}  
Let us say that a dual vector is \emph{full} (or
has full support) if it is positive in every component that is
positive in at least one dual vector. By proposition
\ref{prop:M-comp-dv}, for any full dual vector $\alpha$ and any pure
strategies $c_i$ and $d_i$ in $C_i$, $\alpha_i(d_i|c_i)>0$ if and
only if $d_i$ jeopardizes $c_i$. Existence of full dual vectors
follows from the convexity of the set of dual vectors. A \emph{full
dual reduction} of $\Gamma$ is any reduced game $\Gamma/\alpha$
where $\alpha$ is a full dual vector.

The interest of full dual reductions is not only that there are less
full dual reductions than dual reductions. It also lies in the
following properties, which remain implicit in Myerson: any full
dual vectors, having the same positive components, induce the same
minimal absorbing sets. Therefore, in any full dual reduction, the
pure strategies which are eliminated (resp. kept as pure strategies,
grouped together) are the same. Moreover, in any full dual
reduction, there are weakly less pure strategies remaining than in
any other dual reduction of the same game. That is, if $\alpha$ is a
full dual vector and $\beta$ is any dual vector, then
$|C_i/\alpha_i| \leq |C_i/\beta_i|$. This follows from the fact that
any component that is positive in $\beta_i$ is also positive in
$\alpha_i$, and from basic properties of Markov chains.
%
%

\section{Some applications of dual reduction}
\label{sec:app}
This section aims at showing, by way of examples, that dual reduction is a useful tool to study properties of Nash equilibria and correlated equilibria.\\

\noindent \textbf{A unique correlated equilibrium is a Nash
equilibrium: an elementary proof}  Consider first the fact that, if
a game has a unique correlated equilibrium, then this correlated
equilibrium is a Nash equilibrium. Of course, as existence of
correlated equilibria, this follows from the existence of Nash
equilibria and the fact that any Nash equilibrium is a correlated
equilibrium. But, as for existence of correlated equilibria (Hart
and Schmeidler, 1989; Nau and McCardle, 1990), it would be nice to
find a more direct proof, relying solely on linear programming. Dual
reduction provides such a proof.

Indeed, let $\Gamma$ be a game with a unique correlated equilibrium.
By iterative dual reduction, $\Gamma$ may be reduced to an
elementary game $\Gamma^{e}$. Since $\Gamma^{e}$ is elementary (i.e.
has a strict correlated equilibrium with full support), it follows
that either $\Gamma^{e}$ has an infinity of correlated equilibria,
or $\Gamma^{e}$ has a unique strategy profile. Since $\Gamma$ has a
unique correlated equilibrium and since different correlated
equilibria of $\Gamma^{e}$ induce different correlated equilibria of
$\Gamma$, the first case is ruled out. Therefore, $\Gamma^{e}$ has a
unique strategy profile, hence trivially a Nash equilibrium. This
implies that $\Gamma$ has a Nash equilibrium hence that the unique
correlated equilibrium of $\Gamma$ is a Nash equilibrium.

This proof relies on: a) the definition of dual reduction, which
requires the Minimax theorem and existence of stationary
distributions for finite Markov chains; and b) the fact that any
game may be reduced to an elementary game, which Myerson proved
through the strong complementary property of linear programs. Since
the existence of stationary distributions for finite Markov chains
can be deduced from the Minimax theorem, the above proof relies
solely on linear duality. In particular no fixed point theorem is
used.\footnote{The fact that the existence of stationary
distributions for finite Markov chains can be deduced from the
Minimax theorem is mentioned by Mertens et al (1994, ex. 9, p.41).
To see this, let $M=(m_{ij})$ denote a stochastic matrix (that is,
the $m_{ij}$ are nonnegative and each columns sums to unity).
Applying the lemma of Hart and Schmeidler (1989, p.19) with
$a_{jk}=m_{kj}$ and $u$ a basis vector, we get that there exists a
probability vector $x$ such that $Mx=x$. Since Hart and Schmeidler
prove this lemma via the Minimax theorem,
this shows that existence of stationary distributions for finite Markov chains can indeed be deduced from the Minimax theorem.}\\

\noindent {\bf Equilibria with special properties. } Dual reduction also allows us to prove existence of equilibria with special properties. For instance:
\begin{proposition}%
\label{prop:applqs}%
In any finite game, there exists a Nash equilibrium $\sigma$ such
that, for every player $i$ and every pure strategy $c_i$ with
marginal probability zero in all correlated equilibria, $c_i$ is not
a best response to $\sigma_{-i}$.
\end{proposition}
We first need to introduce a new class of dual vectors.\\

\noindent \emph{Definition. } A dual vector is \emph{strong} if for
every pure strategy profile $c$ in $C$ with probability zero in all
correlated equilibria,
$$D(c,\alpha)=\sum_{i \in N} \left[U_i(c_{-i}, \alpha_i \ast c_i) - U_i(c)\right]>0$$

\noindent (the first equality simply recalls the definition of $D(c,\alpha)$). Existence of strong dual vectors follows from Nau and McCardle's (1990) proof of existence of correlated equilibria. Due to the linearity of the conditions defining dual vectors and their refinements, any strictly convex combination of a full dual vector and of a strong dual vector is a dual vector which is both strong and full. This implies that there is at least one strong and full dual vector. It is actually easy to show that a dual vector is both strong and full if and only if it belongs to the relative interior of the (convex) set of dual vectors.

The following lemma is a variant of Myerson's lemma 1:
\begin{lemma}
\label{app-lm:sdv}
Let $i \in N$. Let $\alpha$ be a strong dual vector. Let $\sigma_{-i} \in \times_{j \neq i} \Delta(C_j)$. Assume that for every $j \neq i$, the mixed strategy $\sigma_j$ is $\alpha_j$-stationary. For any pure strategy $c_i$ of player $i$ with marginal probability zero in all correlated equilibria, $U_i(\sigma_{-i},c_i)<U_i(\sigma_{-i}, \alpha_i \ast c_{i})$.
\end{lemma}
\begin{proof}
Let $\sigma=(\sigma_{-i},c_i)$.
Since $c_i$ has marginal probability zero in all correlated equilibria, it follows that every strategy profile $c$ with $\sigma(c)>0$ has probability zero in all correlated equilibria.
By definition of strong dual vectors, this implies that $D(c,\alpha)>0$. Therefore
\[\sum_{c \in C}\sigma(c)D(c,\alpha)>0\]
But since for every $j\neq i$, $\sigma_j$ is $\alpha_j$-stationary,
it follows that equation (\ref{eq2:rev-variant-lemma1}) is
satisfied. Therefore
\[0< \sum_{c \in C}\sigma(c)D(c,\alpha)=U_i(\sigma_{-i}, \alpha_i \ast \sigma_i) - U_i(\sigma)=U_i(\sigma_{-i}, \alpha_i \ast c_i) - U_i(\sigma_{-i}, c_i) \]
(for the last equality, recall that $\sigma_i=c_i$). The result follows.
\end{proof}
We now prove proposition \ref{prop:applqs}: let $\alpha$ be a strong dual vector and $\sigma$ a Nash equilibrium of $\Gamma/\alpha$, hence of $\Gamma$. By construction of $\Gamma/\alpha$, for every $j$ in $N$, $\sigma_j$ is $\alpha_j$-stationary. Therefore, by lemma  \ref{app-lm:sdv}, if $c_i$ has marginal probability zero in all correlated equililibria, then $c_i$ is not a best-response to $\sigma_{-i}$. So $\sigma$ satisfies the desired property. This completes the proof.

A corollary of proposition \ref{prop:applqs} is that, if a game has a unique correlated equilibrium, then this correlated equilibrium is a quasi-strict Nash equilibrium. This can also be shown directly as in
(Viossat, 2007).\\

\noindent \textbf{Tight games} Consider the class of tight games, introduced by Nitzan (2005):
\begin{definition}
A game is \emph{tight} if in every correlated
equilibrium $\mu$, all incentives constraints are tight:
\begin{equation}
\label{eq:deftight} \sum_{c_{-i} \in C_{-i}} \mu (c) [U_i(c_{-i},
d_i)-U_i(c)] = 0 \hspace{5 mm} \forall i \in N, \forall c_i \in C_i,
\forall d_i \in C_i
\end{equation}
\end{definition}
Nitzan (2005) shows that, for any positive integer $n$, there is an open set of tight games within the set of $n \times n$ bimatrix games.
\begin{proposition}
Any tight game has a completely mixed Nash equilibrium
\end{proposition}
\begin{proof}
The definition of tight games may be rephrased as follows: a game is tight if $d_i$ jeopardizes $c_i$, for every player $i$ and every couple of pure strategies $(c_i,d_i)$ in $C_i \times C_i$.
Therefore, if $\alpha$ is a full dual vector of a tight game $\Gamma$, then for every couple of pure strategies $(c_i,d_i)$ in $C_i \times C_i$,
we have $\alpha_i(d_i|c_i)>0$. Thus, for the Markov chain induced by $\alpha_i$, there is a unique minimal absorbing set: the whole of $C_i$.
Therefore, in the full dual reduction $\Gamma/\alpha$, all strategies of player $i$ are grouped together, and this for every $i$ in $N$.
It follows that $\Gamma/\alpha$ has a unique strategy profile $\sigma$, which is a completely mixed strategy profile of $\Gamma$; furthermore,
since $\sigma$ is trivially a Nash equilibrium of $\Gamma/\alpha$, it is a Nash equilibrium of $\Gamma$.\end{proof}

\noindent \textbf{A class of games introduced by Nau et al (2004). }
Recall that $C_i^c$ denotes the set of pure strategies of player $i$
with positive marginal probability in at least one correlated
equilibrium. Let us say that a game is \emph{pretight} if, in every
correlated equilibrium $\mu$, every incentive constraint
(\ref{eq:defcor}) with $c_i$ and $d_i$ in $C_i^c$ is tight. That is,
\begin{equation}
\sum_{c_{-i} \in C_{-i}} \mu (c) [U_i(c_{-i}, d_i)-U_i(c)] = 0 \quad
\forall i \in N, \forall c_i \in C_i^c, \forall d_i \in C_i^c
\end{equation}
This condition, which is weaker than (\ref{eq:deftight}), has been
introduced by Nau et al (2004). Recall that the inequalities
defining the set of correlated equilibria are linear in $\mu$, so
that the set of correlated equilibria is a convex polytope. Nau et
al (2004, proposition 2) showed that in any finite
game, all Nash equilibria belong to the relative boundary of this
polytope unless: (a) the game is pretight, and (b) there exists a
Nash equilibrium with support $\times_{i \in N} C_i^c$. Dual
reduction allows to show that (a) implies (b).
\begin{proposition}
\label{prop:pretight}
Any pretight game has a quasi-strict Nash equilibrium with support $\times_{i \in N} C_i^c$.
\end{proposition}
The proof of proposition \ref{prop:pretight} uses results from section \ref{sec:prop},
and so we postpone it to the appendix. Proposition \ref{prop:pretight} does not only show that Nau et al's condition (a)
implies their condition (b), but also that any pretight game has a quasi-strict Nash equilibrium:
a nontrivial fact, since for $n \geq 3$, not all $n$-player games have a quasi-strict Nash equilibrium (van Damme, 1991).
Moreover, from proposition \ref{prop:pretight}, a converse of Nau et al's result may be obtained.
Namely, if a game is pretight, then there is a Nash equilibrium in the relative interior of the
correlated equilibrium polytope (Viossat, 2006).\\

\noindent \textbf{Dimension of the correlated equilibrium polytope. } We now turn to another property of the correlated equilibrium polytope. It is well known that in generic $2 \times 2$ games,
the correlated equilibrium polytope has dimension 0 or 3, but not 1 nor 2. This might seem a peculiarity of  $2 \times 2$ games, but dual reduction allows to show that in generic finite games, there are dimensions that the correlated  equilibrium polytope cannot have:
\begin{proposition}
\label{app-prop:geom}
In a generic finite game in which at least two players have at least two pure strategies, the correlated equilibrium polytope does not have dimension $|C|-2$.\footnote{Here is a more formal statement:
let $n \geq 2$, let $(m_1,...,m_n) \in \N^n$ and assume that $\prod_{j \neq i} m_j \geq 2$ for all $i$ in $\{1,...,n\}$.
Then within the set of $m_1 \times ...\times m_n$ games, the set of games whose correlated equilibrium polytope has dimension $m_1 \times ...\times m_n -2$ has Lebesgue measure $0$.
The restriction to games which are not essentially one-player games is needed, because in $2 \times 1$ games, the correlated equilibrium polytope has generically dimension $0=2-2$.}
\end{proposition}
We first need to show that, unless a pure strategy is dominated or
equivalent to a mixed strategy, there is no dual reduction which
simply consists in eliminating this strategy (the converse has been
shown by Myerson).
\begin{lemma}
\label{lm:weakdom2}%
Let $c_i \in C_i$; assume that there exists a dual vector $\alpha$
such that $c_i \notin C_i/\alpha_i$ and $C_j/\alpha_j=C_j$ for all
$j$ in $N \backslash \{i\}$. Then there exists a mixed strategy
$\sigma_i$ in $\Delta(C_i)$ such that $\sigma_i \neq c_i$ and
$U_i(c_{-i},\sigma_i) \geq U_i(c)$ for all $c_{-i}$ in $C_{-i}$.
\end{lemma}
\begin{proof}
Let $\sigma_i= \alpha_i \ast c_i$. Since for all $j \neq i$, every
pure strategy in $C_j$ is $\alpha_j$-stationary, equation
(\ref{eq2:rev-variant-lemma1}) yields:
$$U_i(c_{-i},\sigma_i) \geq U_i(c) \quad \quad \forall c_{-i}
\in C_{-i}$$
Furthermore, $c_i \notin C_i/\alpha_i$ hence $c_i$ is not $\alpha_i$-stationary. Therefore $\sigma_i \neq c_i$
\end{proof}
We now prove the proposition. For every $i$ in $N$, and every couple of pure strategies $(c_i,d_i)$ in $C_i \times C_i$, let
\[H_{c_i,d_i}:=\left\{ \mu \in \R^{|C|}, \sum_{c_{-i} \in C_{-i}} \mu (c) [U_i(c_{-i}, d_i)-U_i(c)] = 0\right\}\]
Thus, $d_i$ jeopardizes $c_i$ if and only if all correlated equilibria belong to $H_{c_i,d_i}$.

Generically, either a player has a strictly dominated strategy (case 1) or the following properties are satisfied simultaneously (case 2):
\begin{equation}%
\label{eq:nodom}%
\sigma_i \neq c_i \Rightarrow \exists c_{-i} \in C_{-i},
U_i(c_{-i},\sigma_i)<U_i(c), \, \forall i \in N, \forall c_i \in
C_i, \forall \sigma_i \in \Delta(C_i)
\end{equation}
\begin{equation}
\label{app-eq:gen2} H_{c_i,d_i} \neq H_{c_j,d_j}, \quad \forall
i\neq j, \forall (c_i,d_i) \in C_i \times C_i, \forall (c_j,d_j) \in
C_j \times C_j
\end{equation}
so it suffices to examine these two cases.\\

\emph{Case 1}: let $c_i$ be strictly dominated. Then $c_i$ has marginal
probability zero in all correlated equilibria. The correlated
equilibrium polytope may thus be seen as a subset of
$\Delta(\hat{C})$ with $\hat{C}=(C_i\backslash\{c_i\})\times C_{-i}$;
therefore, its dimension is at most $(|C_i|-1) \times |C_{-i}|=|C|-|C_{-i}|-1 < |C|-2$.\\

\emph{Case 2}: if the game is elementary, then the correlated equilibrium
polytope has dimension $|C|-1 \neq |C|-2$. Otherwise, there exists a
player $i$ and a couple of pure strategies $(c_i,d_i)$, with $c_i
\neq d_i$, such that $d_i$ jeopardizes $c_i$. That is, every
correlated equilibrium is in $H_{c_i,d_i}$. Therefore, by
proposition \ref{prop:M-comp-dv}, there exists a dual vector
$\alpha$ such that $c_i \notin C_i/\alpha_i$. Due to condition
(\ref{eq:nodom}) and lemma \ref{lm:weakdom2}, this implies that
there exists $j$ in $N \backslash \{i\}$ and $c_j$ in $C_j$ such
that $c_j \notin C_j/\alpha_j$. It follows that $c_j$ is jeopardized
by some strategy $d_j$ in $C_j \backslash \{c_j\}$.
That is, every correlated equilibrium is in $H_{c_j,d_j}$. Therefore
the correlated equilibrium polytope is a subset of $\Delta(C) \cap
H_{c_i,d_i} \cap H_{c_j,d_j}$. It follows from (\ref{eq:nodom}) that
$H_{c_i,d_i}$ and $H_{c_j,d_j}$ are proper subsets of $\R^{|C|}$,
and from (\ref{app-eq:gen2}) that they are distinct. As an
intersection of two distinct hyperplanes of $\R^{|C|}$, $H_{c_i,d_i}
\cap H_{c_j,d_j}$ is a vector space of dimension $|C|-2$. Its
intersection with the simplex $\Delta(C)$ has at most dimension
$|C|-3$ and it includes the correlated equilibrium polytope.
Therefore, this polytope has at most dimension $|C|-3$. This
completes the proof.

\section{Properties of dual reduction}
\label{sec:prop} The examples of applications given in the previous
section suggest, at least to us, that it might be worthwhile to
study in greater depth the properties of dual reduction. This is the
object of this section. Section \ref{sec:rescaling} shows that,
though positive affine transformations and other rescalings of the
payoffs affect the set of dual vectors, they do not change the ways
in which a game may be reduced.  In section \ref{sec:strat}, we
try to understand which kind of strategies are eliminated in dual reductions.
We show for instance that strategies with zero probability in all correlated equilibria are
eliminated in all full dual reductions, but that this is not true of weakly dominated strategies.
Section \ref{sec:classes} studies the properties of dual reduction in specific classes of games such as zero-sum
games or symmetric games. Finally, section \ref{sec:uniqueness}
shows that, in almost all $2$-player games, the process of iterative
full dual reduction is uniquely defined.

\subsection{Payoff rescaling and dual reduction}
\label{sec:rescaling}
Let $\Gamma$ and $\Gamma'$ be two games with the same sets of
players and strategies. Let  $U'_i$ 
denote the utility function of player $i$ in $\Gamma'$.
\begin{definition}
\label{def:rescaling}%
The game $\Gamma'$ is a \emph{rescaling} of $\Gamma$ if for every player
$i$ in $N$, there exists a positive constant $a_i$ and a function $f_i: C_{-i} \to \mathbb{R}$ such that: %
$$ U'_i(c)=a_i.U_i(c) + f_i(c_{-i}), \hspace{0.5 cm} \forall c \in
C$$
\end{definition}
If $\Gamma'$ is a rescaling of $\Gamma$, then $\Gamma$ and $\Gamma'$ represent essentially the same economic situation. So it is reassuring to note that they can be reduced in the same ways:
\begin{proposition}
\label{prop:rescaling} Let $\Gamma'$ be a rescaling of $\Gamma$. If
$\alpha$ is a dual vector of $\Gamma$ then there exists a dual vector
$\alpha'$ of $\Gamma'$ such that $\Gamma'/\alpha'=\Gamma'/\alpha$. That is, $\Gamma'/\alpha$ is a dual reduction of $\Gamma'$.
\end{proposition}
This result would be obvious if a game and its rescalings had the same dual vectors, but the following example shows that this is not the case.
\begin{example}
\label{ex:rescaling} $$
\begin{array}{ccc}
        \begin{array}{ccc}
                            & x_2  &  y_2 \\
                        x_1 & 1,-1 & -1,1    \\
                        y_1 & -1,1  & 1,-1
       \end{array}
& \hspace{2 cm} &
        \begin{array}{ccc}
                            & x_2  &  y_2 \\
                        x_1 & 2,-1 & -2,1    \\
                        y_1 & -2,1  & 2,-1
       \end{array}
\end{array}
$$
Let $\Gamma$ denote the game on the left (Matching Pennies)
and $\Gamma'$ its rescaling on the right. Define $\alpha$ by
$\alpha_i \ast x_i=\alpha_i \ast y_i =\frac{1}{2}x_i+\frac{1}{2}y_i$ for every $i$ in $\{1,2\}$.
Then $\alpha$ is a dual vector for $\Gamma$ but not for $\Gamma'$.
\end{example}
The key is that different vectors of deviation plans may induce the
same reduced game. For instance, in example  \ref{ex:rescaling},
define $\alpha'$ by $\alpha'_1 \ast c_1=\frac{1}{2}(c_1 + \alpha_1
\ast c_1)$ for all $c_1$ in $\{x_1,y_1\}$, and $\alpha'_2=\alpha_2$.
It may be checked that $\alpha'$ is a dual vector of $\Gamma'$ and
that $\alpha$ and $\alpha'$ induce the same reduced  game:
$\Gamma'/\alpha'=\Gamma'/\alpha$. Therefore, even though $\alpha$ is
not a dual vector of $\Gamma'$, $\Gamma'/\alpha$ is a dual reduction
of $\Gamma$. We now generalize this idea.
\begin{lemma}
\label{lm:rescaling}
Let $\alpha_i$ be a deviation plan for
player $i$. Let $0<\epsilon \leq 1$. Define the deviation plan $\alpha_i^{\epsilon}$ by
$\alpha_i^{\epsilon} \ast c_i =\epsilon (\alpha_i \ast c_i) + (1-\epsilon) c_i$. Then $C_i/\alpha_i=C_i/\alpha_i^{\epsilon}$.
\end{lemma}
\begin{proof}
For any mixed strategy $\sigma_i$ in $\Delta(C_i)$, %
$\alpha_i^{\epsilon} \ast \sigma_i -\sigma_i =\epsilon (\alpha_i
\ast \sigma_i - \sigma_i)$. Therefore, $\alpha_i$ and
$\alpha_i^{\epsilon}$ induce the same stationary strategies
\end{proof}
We can now prove proposition \ref{prop:rescaling}. The proof consists in showing that dual vectors of $\Gamma$ may be ``rescaled" into dual vectors of $\Gamma'$.
Let $\alpha$ be a dual vector of $\Gamma$. Let $a_k=\min_{i\in N} a_i$ and, for each $i$ in $N$, let $\epsilon_i=a_k/a_i$ (the constants $a_i$ are those of
definition \ref{def:rescaling}). Recall the definition of $\alpha_i^{\epsilon}$ from lemma \ref{lm:rescaling} and let $\alpha'=(\alpha_i^{\epsilon_i})_{i\in N}$.
For every $c$ in $C$, let $D'(c,\alpha')=\sum_{i \in N}[U'_i(c_{-i}, \alpha'_{i} \ast c_i)-U'_i(c)]$. A simple computation shows that
\begin{equation}
\label{eq:proofrescaling} D'(c,\alpha')=a_k \times
D(c,\alpha) \geq 0 \quad \forall c \in C
\end{equation}
where the inequality follows from the fact that $\alpha$ is a dual vector of $\Gamma$. Therefore $\alpha'$ is a dual vector of
$\Gamma'$. But lemma \ref{lm:rescaling} implies that
$\Gamma'/\alpha'= \Gamma'/\alpha$. This completes the proof.

\subsection{Which strategies and equilibria are eliminated?}
\label{sec:strat}
We say that  the pure strategy $c_i$ is \emph{eliminated} in the dual reduction $\Gamma/\alpha$ if $\sigma_i(c_i)=0$ for all
$\sigma_i$ in $C_i/\alpha_i$. A pure strategy profile $c$ is eliminated if $\sigma(c)=0$ for all $\sigma$ in
$C/\alpha$; that is, if for some $i$ in $N$, $c_i$ is eliminated. Finally, a correlated equilibrium $\mu$ is eliminated if there is no correlated strategy of $\Gamma/\alpha$ which induces $\mu$ in $\Gamma$. In this section, we try to understand which kind of strategies and equilibria are eliminated in dual reductions.

A first remark is that, as shown by Myerson, if a pure strategy is weakly dominated then
there is a dual reduction which consists in eliminating this strategy. In that sense, dual reduction generalizes elimination of
dominated strategies. However,  this does not mean that weakly
dominated strategies are eliminated in all non-trivial dual
reductions, not even in full dual reductions:
\begin{example}
\label{ex:weakdom}
$$
\begin{array}{ccc}
                & x_2 &  y_2\\
            x_1 & 1,1 & 1,0 \\
            y_1 & 1,0 & 0,0
\end{array}
%
%
$$%
In the above game, $\mu$ is a correlated equilibrium if and only
if $y_2$ is not played in $\mu$. That is,
$\mu(x_1,y_2)=\mu(y_1,y_2)=0$. Therefore $x_1$ and $y_1$ jeopardize each other.
It follows that, in all full dual reductions, $x_1$ and $y_1$ are grouped together, hence $y_1$ is not eliminated.
\end{example}
By contrast, in full dual reductions, strategies that have marginal probability zero in all correlated equilibria are eliminated, and a similar result holds for strategy profiles.

\begin{proposition}
\label{prop:uncoherent}
Let $c \in C$ (resp. $c_i \in C_i$) be a pure strategy profile (resp. pure strategy) with probability zero in
all correlated equilibria. Then $c$ (resp. $c_i$) is eliminated in all full dual reductions and in all elementary iterative dual reductions.
\end{proposition}
\begin{proof}
First consider elementary iterative dual reductions: if $\Gamma^{e}$
is an elementary iterative dual reduction of $\Gamma$, it has a
correlated equilibrium with full support, which induces a correlated
equilibrium of $\Gamma$. Therefore all pure strategy profiles (resp.
pure strategies) of $\Gamma$ that have not been eliminated in $\Gamma^{e}$
have positive probability in some correlated
equilibrium.\footnote{I owe my understanding of this point to Roger
Myerson.}

Now consider full dual reductions. Let $\alpha$ be a dual vector.
Let $\sigma \in C/\alpha$. By definition of $C/\alpha$, $\alpha_i
\ast \sigma_i=\sigma_i$ for all $i$ in $N$. Therefore, it follows
from (\ref{eq:rev-variant-lemma1}) that $\sum_{c \in C} \sigma(c)
D(c,\alpha) = 0$. Since by definition of dual vectors, $D(c,\alpha)
\geq 0$ for all $c$ in $C$, this implies that
$$D(c,\alpha)>0 \Rightarrow \sigma(c)=0.$$
Assume now that $\alpha$ is strong and full, and let $c$ be a pure
strategy profile with probability zero in all correlated equilibria.
Since $\alpha$ is a strong dual vector, we have $D(c,\alpha)>0$
hence $\sigma(c)=0$. Therefore, $c$ is eliminated in the dual
reduction $\Gamma/\alpha$. But since $\alpha$ is a
full dual vector and since in all full dual reductions, the same
strategy profiles are eliminated, it follows that $c$ is eliminated
in all full dual reductions.\footnote{Since all full dual reductions eliminate the same strategies, they
also eliminate the same strategy profiles.}

Finally, let $c_i$ be a pure strategy profile with marginal probability zero in all correlated equilibria.
For all $c_{-i}$ in $C_{-i}$, $c=(c_{-i},c_i)$ has probability zero in all correlated equilibria.
Therefore, in a full dual reduction, $(c_{-i},c_i)$ is eliminated for all $c_{-i}$ in $C_{-i}$, whence $c_i$ itself is eliminated.
\end{proof}

We now turn to equilibria. Let us say that in a dual reduction, a correlated equilibrium $\mu$ is eliminated if there is no correlated strategy of the reduced game which induces $\mu$ in the original game.
\begin{proposition}
Strict correlated equilibria cannot be eliminated, not even in an
iterative dual reduction.
\end{proposition}
\begin{proof}
If $\mu$ is a strict correlated equilibrium, a strategy that has
positive marginal probability in $\mu$ cannot be jeopardized by
another strategy. It follows that in any dual reduction of $\Gamma$,
all pure strategies used in $\mu$ remain as pure
strategies. Furthermore, as the player's options for deviating are
more limited in the reduced game than in $\Gamma$, $\mu$ is a fortiori
a strict correlated equilibrium of the reduced game. Inductively,
in any iterative dual reduction of $\Gamma$, all strategies used in $\mu$ are available and $\mu$ is still a
strict correlated equilibrium.\footnote{The proof shows that a pure strategy that has positive marginal
probability in some strict correlated equilibrium can never be
eliminated nor grouped with other strategies. This generalizes the
fact that elementary games cannot be reduced.}
\end{proof}

By contrast, the following example shows that completely mixed Nash
equilibria may be eliminated in all nontrivial dual reductions:

$$    \begin{array}{cccc}
                            & x_2  &  y_2  & z_2  \\
                        y_1 & 2,0  & 1,1   & -1,-1 \\
                        z_1 & 2,0  & -1,-1 &  1,1
       \end{array}
       \quad \quad \quad
       \begin{array}{ccc}
                            &  y_2  &  z_2  \\
                        y_1 &  1,1  & -1,-1 \\
                        z_1 & -1,-1 &  1,1
       \end{array}
$$
In the left game, playing each strategy with equal probability is a
completely mixed Nash equilibrium. However, the unique nontrivial dual reduction is the game on the right, in which $x_2$ and
thus all completely mixed Nash equilibria of the original game have been eliminated.
(To see that the only nontrivial dual reduction consists in eliminating $x_2$, note that,
for player $2$, $x_2$ is equivalent to $\frac{1}{2} y_2 + \frac{1}{2} z_2$; this implies that $y_2$ and $z_2$ jeopardize
$x_2$. Furthermore, for any $i$ in $\{1,2\}$, $y_i$ and $z_i$ must be stationary under any
dual vector because they have positive probability in a strict
correlated equilibrium. The result follows.)

In the above example, the reduced game is obtained by eliminating a pure strategy of player $2$ which
is redundant in the sense that it yields the same payoffs to player $2$ as another (mixed) strategy.
More generally, let $\Gamma'=\{N, (C'_i)_{i \in N}, (U_i)_{i \in N}\}$ be a game built on $\Gamma$
by omitting some pure strategies. Assume that the omitted strategies are redundant in the following sense:
\begin{equation}
\label{cond:rnf}
\forall i \in N, \forall c_i \in C_i\backslash C'_i, \exists \sigma_i \in \Delta(C'_i), \forall c_{-i} \in C_{-i}, U_i(c_{-i}, c_i)=U_i(c_{-i}, \sigma_i)
\end{equation}
Then $\Gamma'$ is a dual reduction of $\Gamma$. (Indeed, for all $i$ in $N$,
let  $\alpha_i \ast c_i=\sigma_i$ (defined in (\ref{cond:rnf}))  if $c_i \notin C'_i$ and $\alpha_i \ast c_i=c_i$ otherwise;
this defines a dual vector $\alpha$ such that $\Gamma/\alpha=\Gamma'$.) In particular, dual reduction allows to reduce any game into its reduced normal form (Kohlberg and Mertens, 1986).
%

We conclude this section by noting that the links between dual
reduction, elimination of unacceptable strategies (Myerson, 1986)
and perfect correlated equilibria (Dhillon and Mertens, 1996) are
explored in (Viossat, 2005).

\subsection{Dual reduction in specific classes of games}
\label{sec:classes}
\subsubsection{Two-player zero-sum games}
\label{sec:zs}
It is easy to show that any dual reduction of a zero-sum game is a zero-sum game with the same value.
Furthermore, dual vectors may be built easily from optimal
strategies:
\begin{proposition} \label{prop:zerosum1} Let $\Gamma$
denote a two-player zero-sum game and $\alpha$ a profile of deviation plans.
If for all $i=1,2$ and for all $c_i$ in $C_i$, the mixed strategy $\alpha_i \ast c_i$ is an optimal strategy,
then $\alpha$ is a dual vector.
\end{proposition}
\begin{proof}
Let $c$ be a pure strategy profile. Since $\alpha_1 \ast c_1$ is
optimal, it follows that $U_1(\alpha_1 \ast c_1, c_2) \geq v$,
where $v$ is the value of the game. Similarly, $U_2(c_1, \alpha_2
\ast c_2) \geq -v$. Therefore, $U_1(\alpha_1 \ast c_1, c_2) +
U_2(c_1, \alpha_2 \ast c_2) \geq 0=
U_1(c)+U_2(c)$.
Since this holds for all $c$ in $C$, it follows that $\alpha$ is a
dual vector.
\end{proof}
\begin{corollary}
\label{cor:zerosum1}%
For every Nash equilibrium $\sigma$ of a zero-sum game, there
exists a dual reduction in which the reduced
set of strategy profiles is the singleton $\{\sigma\}$.%
 \end{corollary}
\begin{proof}
In the particular case of proposition \ref{prop:zerosum1}
where $\alpha_i \ast c_i=\sigma_i$ for every player $i$ and every
$c_i$ in $C_i$, the only $\alpha_i$-stationary strategy is
$\sigma_i$. Therefore, the reduced set $C_i/\alpha_i$ of pure
strategies of player $i$ is the singleton $\{\sigma_i\}$.
\end{proof}
\begin{proposition}
\label{prop:zerosum2} %
Consider a two-player zero-sum game. In all full dual reductions, for every player, all pure strategies with positive marginal probability in at least one correlated equilibrium
are grouped together and all other pure strategies are eliminated.
\end{proposition}
\begin{proof}
Let $C_i^c$ denote the set of pure strategies of player $i$ with positive probability in at least one correlated equilibrium. Consider a full dual reduction. The fact that strategies in $C_i\backslash C_i^c$ are eliminated follows from proposition \ref{prop:uncoherent}. Furthermore, it follows from (Forges, 1990) that player $i$ has an optimal strategy with support $C_i^c$. Due to propositions \ref{prop:zerosum1} and \ref{prop:M-comp-dv}, this implies that all strategies in $C_i^c$ jeopardize each other. Therefore, in a full dual reduction, they are either all eliminated or all grouped together. The first case is ruled out because all other strategies are eliminated. This completes the proof.
\end{proof}
\subsubsection{Games with a unique correlated equilibrium}
\label{sec:uncor}
If $\Gamma$ has a unique Nash equilibrium $\sigma$, then any iterative
dual reduction
 of $\Gamma$ has a unique Nash equilibrium, which induces
$\sigma$ in $\Gamma$; but the strategy space need not be reducible to
$\sigma$. In particular, it may be that a (nontrivial) game has a
unique, pure Nash equilibrium but is nevertheless elementary,
hence cannot be reduced. This is the case in example 4 of Nau and McCardle (1990).
By contrast:
\begin{proposition}
\label{prop:uniquecoreq} Assume that $\Gamma$ has a unique
correlated equilibrium $\sigma$ (note that $\sigma$ is then a Nash
equilibrium, hence may be seen as a mixed strategy profile).
In any full or elementary iterative dual reduction, the set of pure strategy profiles is reduced to $\{\sigma\}$.
In particular, $\Gamma$ has a unique full dual reduction.
\end{proposition}
\begin{proof}
The part of the proposition concerning elementary iterative dual reductions follows from our elementary proof of the fact that a unique correlated equilibrium is a Nash equilibrium (section \ref{sec:app}).
For full dual reductions, let $C_i^c$ denote the set of pure strategies of player $i$ with positive probability in at least one correlated equilibrium.
Since there is a unique correlated equilibrium, which is thus a Nash equilibrium, all strategies in $C_i^c$ trivially jeopardize each other. Then apply the same argument as in the proof of proposition \ref{prop:zerosum2}.
\end{proof}

\subsubsection{Games with symmetries}
\label{sec:sym}
Let $\Gamma$ be a two-player symmetric game. That is, $C_1=C_2=\{1,2,...,m\}$ and for all $(k,l)$ in $\{1,2,...,m\}^2$, $U_1(k,l)=U_2(l,k)$. If $\alpha=(\alpha_1,\alpha_2)$ is a dual vector, then so are $\alpha'=(\alpha_2,\alpha_1)$ and $\bar{\alpha}=(\alpha + \alpha')/2$. Moreover, $\Gamma/\bar{\alpha}$ is symmetric. This section shows that, more generally, if a game has some symmetries (e.g. cyclic symmetry), then it may be reduced in a way which respect these symmetries.

Let $P$ be a set of permutations of the set of players. Let $\Gamma$ be
a game for which, for every $i$ in $N$ and every
permutation $p$ in $P$, player $i$ and player $p(i)$ have the same set of pure strategies. For every $c$ in $C$ and every permutation $p$ in $P$, define the pure strategy profile $c^p$ by $c^p_{p(i)}=c_i$ for every $i$ in $N$.
\begin{definition}
The game $\Gamma$ is \emph{$p$-symmetric} if, for all $i$ in $N$, $U_{p(i)}(c^p)=U_i(c)$. It is \emph{$P$-symmetric} if it is $p$-symmetric for
every $p$ in $P$.
\end{definition}
\begin{proposition}
\label{prop:symgames} If $\Gamma$ is $P$-symmetric
then there exists a strong and full dual vector $\alpha$ such that
$\Gamma/\alpha$ is $P$-symmetric.
\end{proposition}
\begin{proof}
Without loss of generality, assume that $P$ is the largest set of
permutations such that $\Gamma$ is $P$-symmetric. Note that for any permutations $p$ and $\tilde{p}$, if $\Gamma$ is both $p$- and $\tilde{p}$-symmetric then it is $\tilde{p} \circ p$-symmetric. Therefore, by maximality of $P$,
\begin{equation}
\label{eq:PInv} \forall \tilde{p} \in P, \{\tilde{p} \circ p,
p \in P\}=P
\end{equation}
Let $\alpha$ be a strong and full dual vector. For
every permutation $p$ in $P$, define the dual vector $\alpha^p$ by $\alpha^p_{p(i)}=\alpha_i$ for every $i$ in $N$. Let
$$\bar{\alpha}=\frac{\sum_{p \in P} \alpha^p}{|P|}$$
Let $p$ in $P$. Due to (\ref{eq:PInv}), $\bar{\alpha}^p=\bar{\alpha}$, so that for every $i$ in $N$, $\bar{\alpha}_{p(i)}=\bar{\alpha}^p_{p(i)}=\bar{\alpha}_i$.
It follows that for any $i$ in $N$, $C_{p(i)}/\bar{\alpha}_{p(i)}=C_i/\bar{\alpha}_i$ and that $\Gamma/\bar{\alpha}$ is $p$-symmetric. Therefore $\Gamma/\bar{\alpha}$ is $P$-symmetric.
Finally, since $\alpha$ is a strong and full dual vector, so is $\bar{\alpha}$. This concludes the proof.
\end{proof}
\subsubsection{Generic $2 \times 2$ games}
\label{sec:gen}
Let $\Gamma$ be a $2 \times 2$ game such that
$$c \neq c' \Rightarrow U_i(c) \neq U_i(c'), \hspace{0.5cm}
\forall (c,c') \in C^2, \forall i \in \{1,2\}$$
It is well known that such a game is either elementary, in which case it cannot be reduced, or has a unique correlated equilibrium distribution, in which case proposition \ref{prop:uniquecoreq} applies. The first case corresponds to coordination-like games, with three Nash equilibria: two pure and one completely mixed; the second case to games with either a
dominating strategy or a unique, completely mixed Nash equilibrium. The games in Myerson's figures 3 and 5 are instances of the second case.

\subsection{Uniqueness of the reduction process}
\label{sec:uniqueness}
Let $\Gamma$ denote the rather trivial game:
$$
\begin{array}{ccc}
                & x_2 &  y_2\\
            x_1 & 1, 1 & 0,1 \\
\end{array}
$$
Let $\epsilon \in ]0,1[$. Define the full dual vector $\alpha^{\epsilon}$ by $\alpha^{\epsilon}_2 \ast x_2 = \alpha^{\epsilon}_2 \ast y_2 = \epsilon x_2 + (1-\epsilon) y_2$. In the full dual reduction $\Gamma/\alpha_{\epsilon}$, there is a unique pure strategy profile whose payoffs $(\epsilon, 1)$ depend on $\epsilon$. Thus, even if only full dual reductions are used, there might still be multiple ways to reduce a game. Other examples (omitted) suggest however that multiplicity of full dual reductions typically arises when a player is indifferent between some of his strategies, or becomes so after elimination of strategies of the other players. Such indifference is a non-generic phenomenon and we show below that almost all two-player games have a unique sequence of iterative full dual reductions. We first show that there are severe restrictions on the ways strategies may be grouped together.

{\bf Notation}: for all $i$ in $N$, let $B_i \subset C_i$ and let $B=\times_{i \in N} B_i$.
We denote by $\Gamma_B=(N,(B_i)_{i \in N},(U_i)_{i \in N})$ the game obtained from $\Gamma$ by
reducing the pure strategy set of player $i$ to $B_i$, for all $i$ in $N$.
\begin{proposition}
\label{prop:Nashonblocks}%
Let $\alpha$ be a dual vector. For each $i$ in $N$, let $B_i
\subset C_i$ denote a minimal $\alpha_i$-absorbing set and
$B=\times_{i \in N} B_i$. Let $\sigma_{B_i}$ denote the unique
$\alpha_i$-stationary strategy of player $i$ with support in $B_i$
and
$\sigma_B=( \sigma_{B_i})_{i \in N}$.
We have: $\sigma_B$ is a completely mixed Nash equilibrium of $\Gamma_B$.
\end{proposition}
\begin{proof}
The proof draws on the remark made by Myerson at the end of the proof of his lemma 2. Since $B_i$ is $\alpha_i$-absorbing, we may define $\alpha'_i : B_i \to \Delta(B_i)$ by $\alpha'_i \ast c_i =\alpha_i \ast c_i$ for all $c_i$ in $B_i$. Since $\alpha$ is a dual vector of $\Gamma$, it follows that $\alpha'$ is a dual vector of $\Gamma_B$.  Moreover, $B/\alpha'=\{\sigma_B\}$, hence $\sigma_B$ is a Nash equilibrium of $\Gamma_B/\alpha'$. This implies that $\sigma_B$ is a Nash equilibrium of $\Gamma_B$. Finally, by minimality of $B_i$, the support of $\sigma_{B_i}$ is exactly $B_i$ 
so $\sigma_B$ is completely mixed.
\end{proof}
\begin{corollary}
\label{cor:Nash1}
Assume that for every product $B= \times_{i \in N} B_i$ of subsets $B_i$ of
$C_i$, $\Gamma_B$ has at most one completely mixed Nash
equilibrium. Then $\Gamma$ has a  unique full dual reduction.
\end{corollary}
\begin{proof}
Let $\alpha$ and $\alpha'$ be two full dual vectors. Let $\sigma \in C/\alpha$.
Let $B_i$ denote the support of $\sigma_i$ (seen as an element of $\Delta(C_i)$)
and let $B=\times_{i\in N} B_i$. Note that, as the support of
an $\alpha_i$-stationnary strategy, $B_i$ is a minimal $\alpha_i$-absorbing set.
Since full dual vectors have the same minimal absorbing sets, it follows that $B_i$ is also a minimal $\alpha'_i$-absorbing set.
Therefore, there exists $\tau$ in $C/\alpha'$ such that $\tau_i$ has support $B_i$,
for all $i$ in $N$. By proposition \ref{prop:Nashonblocks}, both $\sigma$ and $\tau$ are
completely mixed Nash equilibria of $\Gamma_B$. By assumption, this implies
$\sigma=\tau$, hence $\sigma \in C'/\alpha$. Therefore $C/\alpha =C/\alpha'$. \end{proof}

In the remainder of this section, $\Gamma$ is a two-player game. Consider the following conditions (which are satisfied by almost all two-player games) :

(a) for all Nash equilibria $\sigma$, the supports of $\sigma_1$ and $\sigma_2$ have the same number of elements\footnote{Any game which is nondegenerate in the sense of von Stengel (2002, def. 2.6 and thm 2.10)
satisfies this condition.}

(b) any game obtained from $\Gamma$ by deleting some pure strategies satisfies (a)

(c) for any $i$ in $\{1,2\}$, for any $B_i \subset C_i$, $B_{-i} \subset C_{-i}$ and $B'_{-i} \subset C_{-i}$,
with $|B_i|=|B_{-i}|=|B'_{-i}| \geq 2$ and $B_{-i} \cap B'_{-i}=\emptyset$, we have:
if $\sigma$ and $\sigma'$ are completely mixed Nash equilibria of $\Gamma_{B_i \times B_{-i}} $
and $\Gamma_{B_i \times B'_{-i}}$, respectively, then $\sigma_i \neq \sigma'_i$.
That is, the same mixed strategy of player $i$ cannot be the $i$-component of a completely mixed Nash equilibrium both in $\Gamma_{B_i \times B_{-i}} $ and in $\Gamma_{B_i \times B'_{-i}}$.
\begin{proposition}
\label{prop:uniquefdr} If $\Gamma$ satisfies condition (b), then $\Gamma$ has a unique full dual reduction.
\end{proposition}
\begin{proof}
It suffices to show that condition (b) implies the assumption of corollary \ref{cor:Nash1}. The proof is by contradiction. Assume that there exists $B = B_1 \times B_2 \subset C_1 \times C_2 $ such that $\Gamma_B$ has  two
distinct completely mixed Nash equilibria $\sigma$ and $\tau$. Without loss of generality, assume $\sigma_1\neq \tau_1$.
There exists $\lambda$ in $\R$ such that $\sigma_1^{\lambda}:=\lambda \sigma_1 + (1-\lambda)\tau_1$ is in $\Delta(C_1)$
but its support is a strict subset of the support of $\sigma_1$. Since $\Gamma_B$ is a bimatrix game
and $\sigma$ and $\tau$ are completely mixed (in $\Gamma_B$), it follows that $(\sigma^{\lambda}_1, \sigma_2)$
is a Nash equilibrium of $\Gamma_B$. But so is $\sigma$. Therefore, for at least one of these equilibria,
the supports of the strategies of the players do not have the same number of elements, hence  condition (b) is not satisfied.
\end{proof}
\begin{proposition}
\label{prop:3poss} If $\Gamma$ satisfies conditions (b) and (c), then there are only three possibilities:
\begin{enumerate}
\item[1]
$\Gamma$ is elementary
\item[2]
In all dual reductions of $\Gamma$, some strategies are eliminated,
but no strategies are grouped together.
\item[3]
In any full dual reduction of $\Gamma$, the reduced strategy space $C/\alpha$ is a singleton.
\end{enumerate}
\end{proposition}
\begin{proof}
Assume that $\Gamma$ is not elementary and let $\alpha$ be a nontrivial dual vector. Assume that some
strategies of player 1 (for instance) are grouped together. That is, there exists a minimal
$\alpha_1$-absorbing set $B_1$ with at least two elements.
Let $B_2$ and $B'_2$ be minimal $\alpha_2$-absorbing sets. Let
$\sigma_{B_1}$ denote the $\alpha_1$-stationary strategy with
support in $B_1$. Define $\sigma_{B_2}$ and $\sigma_{B'_2}$
similarly. By proposition \ref{prop:Nashonblocks},
$(\sigma_{B_1}, \sigma_{B_2})$ and $(\sigma_{B_1}, \sigma_{B'_2})$ are
Nash equilibria of $\Gamma_{B_1 \times B_2}$ and $\Gamma_{B_1 \times B'_2}$, respectively.
Due to conditions (b) and (c), this implies $B_2=B'_2$. Therefore, there is a unique minimal $\alpha_2$-absorbing set: $B_2$. That is, $C_2/\alpha_2$ is a singleton.
 Moreover, due to condition (b), $B_1$ and $B_2$ have the same number of elements. Thus $B_2$ has at least two elements.
 Therefore, by the above reasoning, $C_1/\alpha_1$ is also a singleton and we are done.
\end{proof}
If $\Gamma$ satisfies condition (b), then any game obtained by deleting some pure strategies of
$\Gamma$ obviously satisfies condition (b). Therefore, proposition \ref{prop:3poss} implies that:
\begin{corollary}
\label{cor:locgen} If $\Gamma$ satisfies conditions (b) and (c), then so does any dual reduction of $\Gamma$.
\end{corollary}
Putting proposition \ref{prop:uniquefdr} and corollary \ref{cor:locgen}  together, we get
\begin{theorem}
If $\Gamma$ satisfies conditions (b) and (c), then $\Gamma$ has a
unique sequence of iterative full dual reductions.
\end{theorem}
\begin{appendix}
\section{Proof of proposition \ref{prop:pretight}}
Let $\Gamma$ be pretight. Let $\alpha$ be a strong and full dual vector. Note that by definition of pretight games, for any $i$ in $N$,
all strategies of $C_i^c$ jeopardize each other. Since $\alpha$ is full, this implies that in $\Gamma/\alpha$ there is a unique pure strategy profile $\{\sigma\}$, which is a Nash equilibrium of $\Gamma$ with support $\times_{i \in N} C_i^c$ (the argument is the same as in the proof of
proposition \ref{prop:zerosum2}). Moreover, since $\alpha$ is strong, it follows from the proof of lemma \ref{app-lm:sdv} that any pure strategy in $C_i\backslash C_i^c$ is not a best-response to $\sigma_{-i}$. Since $\sigma_i$ has support $C_i^c$, this implies that $\sigma$ is quasi-strict.
\end{appendix}

\end{document}